\documentclass{amsart}
\usepackage{amsmath}
\usepackage{amsfonts}
\usepackage{amssymb}
\usepackage{amsthm}
\usepackage[T1]{fontenc}
\usepackage{srcltx}
\usepackage[utf8]{inputenc}

\def\Z{\mathbb{Z}}

\def\N{\mathbb{N}}
\def\F{\mathbb{F}}

\begin{document}
\title[Lower bounds on the class number]{Lower bounds on the class number of 
algebraic function fields defined over any finite field}
\author{S. Ballet}
\address{Institut de Math\'{e}matiques
de Luminy\\ case 930, F13288 Marseille cedex 9\\ France}
\email{ballet@iml.univ-mrs.fr}

\author{R. Rolland}
\address{Institut de Math\'{e}matiques
de Luminy\\ case 930, F13288 Marseille cedex 9\\ France}
\email{robert.rolland@acrypta.fr}
\date{\today}

\newtheorem{theoreme}{Theorem}[section]
\newtheorem{lemme}[theoreme]{Lemma}
\newtheorem{proposition}[theoreme]{Proposition}
\newtheorem{corollaire}[theoreme]{Corollary}
\theoremstyle{definition}
\newtheorem{conclusion}[theoreme]{Conclusion}
\newtheorem{definition}[theoreme]{Definition}
\newtheorem{remarque}[theoreme]{Remark}
\newtheorem{exemple}[theoreme]{Example}
\renewcommand{\theequation}{\arabic{equation}}
\setcounter{equation}{0}
\begin{abstract}
We give lower bounds on the number of effective divisors of degree $\leq g-1$ 
with respect to the number of places of certain degrees of an
algebraic function field of genus $g$ defined over a finite field. 
We deduce lower bounds and asymptotics for the class number, 
depending mainly on the number of places of a certain degree. 
We give examples of towers of algebraic function fields having 
a large class number. 
\end{abstract}

\maketitle

\section{Introduction}
\label{sec:intro}

\subsection{General context}\label{subsec:gc}

We recall that the class number $h(F/{\mathbb{F}}_q)$ of an algebraic function field $F/{\mathbb{F}}_q$  
defined over a finite field ${\mathbb{F}}_q$ is the cardinality of the Picard group of $F/{\mathbb{F}}_q$.
 This numerical invariant corresponds to the number of ${\mathbb{F}}_q$-rational 
points of the Jacobian of any curve $X({\mathbb{F}}_q)$  
having $F/{\mathbb{F}}_q$ as algebraic function field. 

Estimating the class number of an algebraic function field is a classic problem. 
By the standard estimates deduced from the results of Weil \cite{weil1} \cite{weil2}, we know that 
$$(\sqrt{q}-1)^{2g}\leq h(F/{\mathbb{F}}_q) \leq (\sqrt{q}+1)^{2g},$$
where $g$ is the genus of $F/{\mathbb{F}}_q$. Moreover, these estimates  hold 
for any Abelian variety. 
Finding good estimates for the class number   $h(F/{\mathbb{F}}_q)$ is a difficult problem.
For $g=1$, namely for elliptic curves, the class number ${\mathbb{F}}_q$ is the number of 
${\mathbb{F}}_q$-rational points 
of the curve and this case has been extensively studied. So, from now on we assume that $g\geq 2$.
In \cite{lamd}, Lachaud - Martin-Deschamps obtain the following result: 

\begin{theoreme}\label{theolamd}
Let $X$ be a smooth absolutly  irreducible projective algebraic curve of genus $g$ defined over $\F_q$; 
Let $J_X$ be the Jacobian of $X$ and $h=\mid J_X ({\mathbb{F}}_q)\mid$. Then: 
\begin{enumerate}
\item $$h\geq q^{g-1}\frac{(q-1)^2}{(q+1)(g+1)};$$
\item $$h\geq (\sqrt{q}-1)^2\frac{q^{g-1}-1}{g}\frac{\mid X({\mathbb{F}}_q)\mid+q-1}{q-1};$$
\item if $g>\frac{\sqrt{q}}{2}$ and if $X$ has at least a rational point over ${\mathbb{F}}_q$ , then 
$$h \geq (q^g-1)\frac{q-1}{q+g+gq}.$$
\end{enumerate}
\end{theoreme}
This result was improved in certain cases by the following result of M. Perret in \cite{perr}:
\begin{theoreme}\label{theoperr}
Let $J_X$ the Jacobian variety of the projective smooth irreducible curve $X$
of genus $g$ defined over $\F_q$. Then
$$
\# J_X(\F_q)\geq \left(\frac{\sqrt{q}+1}{\sqrt{q}-1} \right)^{\frac{\# X(\F_q)-(q+1)}{2\sqrt{q}}-2\delta}(q-1)^g
$$
with $\delta=1$ if $\frac{\# X(\F_q)-(q+1)}{2\sqrt{q}} \notin \Z$ and $\delta=0$ otherwise.
\end{theoreme}
Moreover, Tsfasman \cite{tsfa} and Tsfasman-Vladut \cite{tsvl} obtain asymptotics for the Jacobian; 
these best known results can also be found in \cite{tsvlno}. We recall the following three important theorems
contained in this book.
The first one concerns the so-called asymptotically good families.

\begin{theoreme}\label{theoh1}
Let $\{X_i\}_i$ be a family of curves over $\F_q$ (called asymptotically good) of growing genus such that 
$$\lim_{i \rightarrow +\infty} \frac{N(X_i)}{g(X_i)}=A>0,$$
where $N(X_i)$ denotes the number of rational points of $X_i$ over $\F_q$. 
Then $$\liminf_{g \rightarrow +\infty} \frac{log_qh(X_i)}{g(X_i)}\geq 1+A\log_q\left(\frac{q}{q-1}\right).$$
\end{theoreme}

The second important result, established by Tsfasman in \cite{tsfa}, relates to particular families of curves, 
which are called asymptotically exact, namely
the families of (smooth, projective, absolutly irredicible ) curves 
${\mathcal X}/\F_q=\{X_i\}_i$ 
defined over $\F_q$ such that $g(X_i) \rightarrow +\infty$ 
and for any $m\in \N$ the limit 
$$\beta_m=\beta_m({\mathcal X}/\F_q)= \lim_{i\rightarrow +\infty}\frac{B_m(X_i)}{g(X_i)}$$ 
exists, where $B_m(X_i)$ 
denotes the number of points of degree $m$ of $X_i$ over $\F_q$ 
(in term of the dual language of algebraic function field theory, 
it corresponds to the number of places of degree $m$ of  the algebraic function field $F(X_i)/\F_q$ 
of the curve $X_i/\F_q$). 

\begin{theoreme}\label{theoh2}
 For an asymptotically exact family of curves ${\mathcal X}/\F_q=\{X_i\}_i$ defined over $\F_q$, the limit 
  $$H({\mathcal X}/\F_q)=\lim_{i\rightarrow +\infty}{\frac{1}{g(X_i)}\log_q h(X_i)}$$ exists and equals 
  $$H({\mathcal X}/\F_q)=1 +\sum_{m=1}^{\infty}\beta_m.\log_q \frac{q^{m}}{q^{m}-1}.$$
\end{theoreme}
 
The third theorem is a general result concerning the family of all curves defined over $\F_q$. 

\begin{theoreme}\label{theoh3}
Let $\{X_i\}_i$ be the family of all curves over $\F_q$ (one curves from each isomorphism class).
Then, we have $$1\leq \liminf_{g \rightarrow +\infty} \frac{\log_qh(X_i)}{g(X_i)}\leq  \limsup_{g \rightarrow +\infty} 
\frac{\log_qh(X_i)}{g(X_i)}\leq 1+(\sqrt{q}-1)\log_q\left(\frac{q}{q-1}\right).$$
\end{theoreme}

\subsection{New results}\label{subsec:nr}
We remark that Theorems \ref{theolamd} and \ref{theoperr} in the non-asymptotic case and Theorem \ref{theoh1}
in the asymptotic case involve rational points of the curve. But in many cases the number
of rational points is low while the number of places of a certain degree $r$ is large. Then 
it may be interesting to give formulas that are based on the number of points of degree $r\geq 1$. 
This is the main idea of the article, even if the formula applied with $r = 1$ gives in many cases better lower bounds for
the class number
than Theorems \ref{theolamd} and \ref{theoperr}.

In this paper we give lower bounds on the class number of an algebraic function field
of one variable over the finite field $\F_q$ in the two following situations:
\begin{itemize}
 \item in the non-asymptotic case, namely when the function field is fixed; in this context, 
we extend the formulas of Theorem \ref{theolamd} which is given under very weak assumptions, 
to obtain more precise bounds 
taking into account 
the number of points of a given degree $r\geq 1$, possibly of degree one.
 \item in the asymptotic case, where we consider a sequence of function fields $F_k$ of genus $g_k$
growing to infinity; we give asymptotic lower bounds on the class number in several cases, 
under assumptions weaker than those of \cite{tsvlno}; we extend Theorem \ref{theoh1} by using points
of dgree $r$ and improve it by weakening the hypothesis, assuming only that $\frac{B_r(k)}{g_k}$ 
has a strictly positive lower limit, where $B_r(k)$ is the number of places of
degree $r$ of $F_k$; in particular this shows that the lower bound $1$ 
of Theorem \ref{theoh3} can be widely improved if there is at least a place of degree $1$ in each $F_k$
and an integer $r\geq 1$ such that $\liminf_{k \rightarrow +\infty} \frac{B_r(k)}{g_k}>0$;
we also improve Theorem \ref{theoh2} in certain cases by removing the assumption ``$(\F_k)_k$
is an asymptotically exact sequence''.
\end{itemize}

Let $A_n=A_n(F/{\mathbb{F}}_q)$ be the number of effective divisors of degree $n$ of 
an algebraic function field $F/{\mathbb{F}}_q$ 
defined over ${\mathbb{F}}_q$ of genus $g\geq 2$ 
and  $h=h(F/{\mathbb{F}}_q)$ the class number of $F/{\mathbb{F}}_q$. 
Let $B_n=B_n(F/{\mathbb{F}}_q)$ the number of places of degree $n$ of $F/{\mathbb{F}}_q$.

\medskip

Let us set $$S(F/{\mathbb{F}}_q)=\sum_{n=0}^{g-1}A_n+\sum_{n=0}^{g-2}q^{g-1-n}A_n \quad 
\hbox{ and } \quad R(F/{\mathbb{F}}_q)=\sum_{i=1}^{g}\frac{1}{\mid1-\pi_i\mid^2},$$ 
where $(\pi_i,\overline{\pi_i})_{1\leq i \leq g}$ are the reciprocal roots of the numerator 
of the zeta-function $Z(F/{\mathbb{F}}_q,T)$ of $F/{\mathbb{F}}_q$.
By a result due to G. Lachaud and M. Martin-Deschamps \cite{lamd}, we know that
\begin{equation}\label{mainformula}
S(F/{\mathbb{F}}_q)=hR(F/{\mathbb{F}}_q).
\end{equation}

Therefore, to find good lower bounds on  $h$, just find a good lower bound on
$S(F/{\mathbb{F}}_q)$ and  a good upper bound on $R(F/{\mathbb{F}}_q)$. 

It is known by \cite{lamd} that the quantity $R(F/{\mathbb{F}}_q)$  is bounded by the following upper bound: 
\begin{equation}\label{R1}
R(F/{\mathbb{F}}_q)\leq \frac{g}{(\sqrt{q}-1)^2},
\end{equation}
or with this best inquality:  
\begin{equation}\label{R2}
R(F/{\mathbb{F}}_q)\leq \frac{1}{(q-1)^2}  \left(\strut (g+1)(q+1)-B_1(F/{\mathbb{F}}_q)\right).
\end{equation}

Moreover, the inequality (\ref{R2}) is obtained as follows:
$$R(F/{\mathbb{F}}_q)=\sum_{i=1}^{g}\frac{1}{(1-\pi_i)(1-\overline{\pi_i})}=$$
$$\sum_{i=1}^{g}\frac{1}{1+q-(\pi_i+\overline{\pi_i})}.$$ 
Multiplying the denominators by the corresponding conjugated quantities, 
we get: 
$$R(F/{\mathbb{F}}_q)\leq \frac{1}{(q-1)^2}\sum_{i=1}^{g}(1+q+\pi_i+\overline{\pi_i}).$$
This last inequality associated to the following formula deduced from the Weil's formulas:
$$\sum_{i=1}^{g}(\pi_i+\overline{\pi_i})=1+q-B_1(F/{\mathbb{F}}_q)$$
gives the inequality (\ref{R2}). 
This inequality cannot be improved in the general case. 

Hence, in this paper, we propose to study some lower bounds on   
$S(F/{\mathbb{F}}_q)$. 
In this aim, we determine some lower bounds on the number of effective divisors of degree $n\leq g-1$ 
obtained from the number of effective divisors of degree $n\leq g-1$ containing in their support 
a maximum number of places of a fixed degree  $r\geq 1$.
We deduce lower bounds on the class number. 
 It is a successful strategy that allows us 
 to improve the known lower bounds on $h$ in the general case (except if there is no place of degree one). 
 It also allows us to obtain the best known asymptotics for $h$ when we specialize our study
to some families of curves having asymptotically a large number of places of degree $r$ forme some value of $r$, 
namely when $\liminf_{g\rightarrow +\infty} \frac{B_r(g)}{g}>0$.
  
\subsection{Organization of the paper}
In Section \ref{sec:lower} we study the number of effective divisors of an algebraic function 
field over the finite field $\F_q$. As we saw in the introduction, this is the main point to
estimate the class number. By combinatorial considerations, 
we give a lower bound on the number of these divisors. Then we estimate 
the two terms $\Sigma_1=\sum_{n=0}^{g-1}A_n$ and $\Sigma_2=\sum_{n=0}^{g-2}q^{g-1-n}A_n$ of 
the number $S(F/{\mathbb{F}}_q)$ introduced in Paragraph \ref{subsec:nr}.
This yields in Section \ref{sec:class} a lower bound on the class number.
In Section \ref{sec:asymptotic} 
we give the asymptotic behavior of the class number, 
with assumptions weaker than those made in the theorems recalled in the introduction.
The final section \ref{sec:examples} presents several examples where we obtain lower bounds 
on the class number.

\section{Lower bounds on the number of effective divisors}\label{sec:lower}
In this section, we obtain a lower bound on the number of effective divisors of degree $\leq g-1$. 
We derive a lower bounds on $S(F/{\mathbb{F}}_q)$. 


Let $B_r(F/{\mathbb{F}}_q)$ be the number of places of degree $r$.
In the first time, we determine lower bounds on the number of effective divisors of degree $\leq g-1$. 


\medskip

By definition, we have  $S(F/{\mathbb{F}}_q)=\sum_{n=0}^{g-1}A_n+\sum_{n=0}^{g-2}q^{g-1-n}A_n$. 
We will denote by $\Sigma_1$ the first sum of the right member and $\Sigma_2$ the second one:
\begin{equation}\label{sigma}
\Sigma_1=\sum_{n=0}^{g-1}A_n \quad \hbox{and} \quad \Sigma_2=\sum_{n=0}^{g-2}q^{g-1-n}A_n.
\end{equation}
Let us fix a degree $n$ and set
$$U_n=\left \{b=(b_1,\cdots,b_n)~|~ b_i \geq 0 \hbox{ et } \sum_{i=1}^n ib_i = n\right \}.$$
Note first that if $B_i \geq 1$ and $b_i \geq 0$,  the number of solutions of the equation 
$n_1+n_2+ \cdots+ n_{B_i} = b_i$
with integers $\geq 0$ is:
\begin{equation}\label{mcoef}
\left (
\begin{array}{c}
 B_i+b_i-1\\
 b_i
\end{array}
\right ).
\end{equation}
 Then the number of effective divisors of degree $n$ is given by the following result, already mentioned in 
 \cite{tsvl}:

\begin{proposition}\label{propoAn}
The number of effective divisors of degree $n$ of an algebraic function field $F/{\mathbb{F}}_q$ is:
$$A_n=\sum_{b\in U_n}\left [ \prod_{i=1}^n \left (
\begin{array}{c}
 B_i+b_i-1\\
b_i
\end{array}
\right )\right ].$$
\end{proposition}

\begin{proof}
It is sufficient to consider that in the formula, $b_i$ is the sum of coefficients that
are applied to 
the places of degree $i$.
So, the sum of the terms $ib_i$ is the degree $n$ of the divisor. The number of ways to get a divisor of 
degree $ib_i$ with some places of degre $i$ is given by the binomial coefficient (\ref{mcoef}). 
For a  given $b$ , the product of the second member is the number of effective divisors 
for which the weight corresponding to the places of degree $i$ is $ib_i$. Then it remains to compute
the sum over all  possible $b$ to get the number of effective divisors.
\end{proof}

From Proposition \ref{propoAn} we obtain in the next proposition a lower bound on the number of 
effective divisors of degree  $n\leq g-1$. This lower bound is constructed using only
places of degree
$1$ and $r$ (possibly $1$) as follows: the weight associated to places of degree $r$ is maximum, 
namely $\lfloor \frac{n}{r}\rfloor$.
The weight associated to places of degree $1$ is $n \mod r$.

\begin{proposition}\label{minA_n}
Let $r$ and $n$ be two integers $>0$. 
Suppose that $B_1$ and $B_r$ are different from zero. Then
\begin{equation} A_n \geq 
\left (
\begin{array}{c}
 B_r+m_r(n)-1\\
 B_r-1
\end{array}
\right )
\left (
\begin{array}{c}
 B_1+s_r(n)-1\\
 B_1-1
\end{array}
\right )
\end{equation}
where $m_r(n)$ and $s_r(n)$ are respectively the quotient and the remainder 
of the Euclidian division of $n$ by $r$.
\end{proposition}

\begin{proof}
Let $a=(a_i)_{1\leq i \leq n}$ be the following element of $U_n$:
$$a=(s_r(n),0,\cdots 0,m_r(n),0, \cdots, 0).$$ 
Then by Proposition (\ref{propoAn}), we have:
$$A_n=\sum_{b\in U_n} \prod_{i=1}^n 
\left (
\begin{array}{c}
B_i+b_i-1\\
b_i\end{array}\right )\geq  \prod_{i=1}^n \left (\begin{array}{c} B_i+a_i-1\\ a_i\end{array}\right )= $$
$$\left (\begin{array}{c} B_r+m_r(n)-1\\B_r-1 \end{array}\right )\left (\begin{array}{c}B_1+s_r(n)-1\\B_1-1\end{array}\right ).$$
\end{proof}

In particular, if
$r=1$ we obtain:
$A_n \geq 
\left (
\begin{array}{c}
B_1+n-1\\
n
\end{array}
\right )
=
\left (
\begin{array}{c}
B_1+n-1\\
B_1-1 
\end{array}
\right ).
$

\subsection{Lower bound on the sum $\Sigma_1$}

By using  the same lower bound for each term $A_n$ as in Proposition \ref{minA_n},
then by associating the $r$ indexes $n$ with the same $m_r(n)$, we obtain:

\begin{equation}\label{minSommeAn}
\Sigma_1
\geq \sum_{m=0}^{m_r(g-1)-1}\left (\sum_{i=0}^{r-1}
\left (
\begin{array}{c}
B_1+i-1\\
B_1-1
\end{array}
\right )\right)
\left (
\begin{array}{c}
B_r+m-1\\
B_r-1
\end{array}
\right )+
\end{equation}

$$
\left (
\begin{array}{c}
B_r+m_r(g-1)-1\\
B_r-1
\end{array}
\right )
\sum_{i=0}^{s_r(g-1)}
\left (
\begin{array}{c}
B_1+i-1\\
B_1-1
\end{array}
\right )
$$
which gives the following lemma:
\begin{lemme}\label{lem1}
Suppose $B_1>0$ and $B_r>0$, then 
we have the following lower bound on the sum $\Sigma_1=\sum_{n=0}^{g-1}A_n$:
\begin{equation}\label{majsom1r}
\Sigma_1
\geq K_1(r-1,B_1) 
\left (
\begin{array}{c}
 B_r+m_r(g-1)-1\\
 B_r
\end{array}
\right )+
\end{equation}
$$
K_1\left(\strut s_r(g-1),B_1\right)
\left (
\begin{array}{c}
 B_r+m_r(g-1)-1\\
 B_r-1
\end{array}
\right )
$$
where $$K_1(i,B)= 
\left (
\begin{array}{c}
 B+i\\
 B
\end{array}
\right )
.$$
\end{lemme}

\begin{proof}
It is sufficient to apply three times to the second member of 
the inequality (\ref{minSommeAn}) the following combinational formula:
for all the integers $k$ and $l$,
$$\sum_{j=0}^{l}\left (
\begin{array}{c}
 k+j\\
 k
\end{array}
\right )=\left (
\begin{array}{c}
 k+l+1\\
 k+1
\end{array}
\right ).$$
We obtain 
$$
\Sigma_1=\sum_{n=0}^{g-1}A_n 
\geq 
\left (
\begin{array}{c}
 B_1+r-1\\
 B_1
\end{array}
\right ) 
\left (
\begin{array}{c}
 B_r+m_r(g-1)-1\\
 B_r
\end{array}
\right )+
$$
$$
\left (
\begin{array}{c}
 B_r+m_r(g-1)-1\\
 B_r-1
\end{array}
\right )
\left (
\begin{array}{c}
 B_1+s_r(g-1)\\
 B_1
\end{array}
\right ).
$$
\end{proof}

\begin{remarque} 
Note that for any $B_1 \geq 1$ we have $K_1(r-1,B_1) \geq B_1+r-1$, the value $r$ being reached for  
$B_1=1$.
\end{remarque} 

\begin{remarque} 
 If $r=1$ then 
$$m_r(n)=n,~ s_r(n)=0,~ K_1(r-1,B_1)=1,~ K_1\left(\strut s_r(g-1),B_1\right)=1.$$ 
We conclude in this case that:
$$\Sigma_1=\sum_{n=0}^{g-1}A_n \geq 
\left (
\begin{array}{c}
 B_1+g-1\\
 B_1
\end{array}
\right )=
\left (
\begin{array}{c}
 B_1+g-1\\
g-1
\end{array}
\right )=K_1(g-1,B_1).
$$
\end{remarque} 

\subsection{Lower bound on the sum $\Sigma_2$}
 
Now let us study the second sum:
$\Sigma_2=q^{g-1}\sum_{n=0}^{g-2}\frac{A_n}{q^n}$. 
We have:

\begin{equation}\label{S2}
\Sigma_2\geq q^{g-1}\left( \sum_{i=0}^{r-1}\frac{1}{q^{i}}
\left (
\begin{array}{c}
 B_1+i-1\\
 B_1-1
\end{array}
\right )
\right )
\sum_{m=0}^{m_r(g-2)-1} \frac{1}{q^{mr}} 
\left (
\begin{array}{c}
 B_r+m-1\\
 B_r-1
\end{array}
\right )
+
\end{equation}

$$
q^{s_r(g-2)+1}
\left (
\begin{array}{c}
 B_r+m_r(g-2)-1\\
 B_r-1
\end{array}
\right )
\sum_{i=0}^{s_r(g-2)} \frac{1}{q^{i}} 
\left (
\begin{array}{c}
 B_1+i-1\\
 B_1-1
\end{array}
\right ).
$$
In order to give a lower bound on $\Sigma_2$, 
we need to study more precisely the following quantity $Q_r$ which occurs in the previous inequality 
(\ref{S2}):
$$Q_r=\sum_{m=0}^{m_r(g-2)-1} \frac{1}{(q^{r})^m}
\left (
\begin{array}{c}
B_r+m-1\\
 B_r-1
\end{array}
\right )
.$$ 

The value $Q_r$ depends on the parameters $q$, $g$, $r$ and $B_r$.
We give several lower bounds whose accuracy depends on the ranges in which the parameters vary.
Remark that for $g=2$ and $g=3$ we have respectively $Q_r=0$ and $Q_r=1$. 
These lower bounds are specified in the following three lemmas.
Then we can suppose 
in the study of $Q_r$ that
$g>3$.  

\begin{lemme}\label{minoqr}
Let $q$ and $r$ be two integers $>0$ and let $g > 3$ be an integer.
Let $Q_r$ be the sum defined by:
$$Q_r=\sum_{m=0}^{m_r(g-2)-1} \frac{1}{(q^{r})^m}
\left (
\begin{array}{c}
B_r+m-1\\
 B_r-1
\end{array}
\right )
,$$ 
where $m_r(n)$ denotes the quotient of the Euclidian division of $n$ by $r$.
Let us set 
$$f_r=
\left\{
\begin{array}{lcl}
0 & \hbox{ if } & \frac{g-1}{2}<r\leq g-1;\\
1 & \hbox{ if } & r\leq\frac{g-1}{2} \hbox{ and } B_r<q^r;\\
\min\left(\lfloor \frac{B_r-q^r}{q^r-1} \rfloor+1,m_r(g-2)-1\right) & \hbox{ if } 
& r\leq \frac{g-1}{2} \hbox{ and } B_r\geq q^r;
\end{array}
\right.
$$
then  $$Q_r\geq \frac{q^{rm_r(g-2)}-1}{q^{r(m_r(g-2)-1)}(q^r-1)}+\frac{(B_r-1)}{q^r}f_r.$$
\end{lemme}

\begin{proof}
Let us write 
$$Q_r= \sum_{m=0}^{m_r(g-2)-1} \frac{1}{(q^{r})^m}+ \Delta_r,$$
where 
$$\Delta_r= \sum_{m=0}^{m_r(g-2)-1} \frac{1}{(q^{r})^m}
\left[\left (
\begin{array}{c}
B_r+m-1\\
 B_r-1
\end{array}
\right ) -1 \right]
.$$
Then
$$Q_r= \frac{q^{rm_r(g-2)}-1}{q^{r(m_r(g-2)-1)}(q^r-1)} +\Delta_r.$$
First, if $\frac{g-2}{2}<r \leq g-2$
then $m_r(g-2)=0$ or $m_r(g-2)=1$, hence $\Delta_r=0$. Let us suppose now that $r \leq \frac{g-2}{2}$, then
$m_r(g-2)\geq 2$. Hence, the sum $\Delta_r$ has at least the two first terms, namely $0$ (for $m=0$) and 
$B_r-1$ (for $m=1$). Let us set for any $n\geq 0$
$$u_n=\frac{1}{q^{rn}}
\left (
\begin{array}{c}
 B_r+n-1\\
 B_r-1
\end{array}
\right )
,$$
and compute:
$$\frac{u_{n+1}}{u_n}=\frac{1}{q^r}\frac{B_r+n}{n+1}.$$
Let us set $n_0=\lfloor \frac{B_r-q^r}{q^r-1}\rfloor +1$. Then, if
$0 \leq n <n_0$
we have $u_{n+1}-1 \geq u_n -1$.
If $B_r<q^r$, then $n_0\leq 1$. Let us give in this case a lower bound on the sum $\Delta_r$ by 
considering only the sum of the two first terms
i.e $(B_r-1)/q^r$. If $B_r \geq q^r $, then $n_0 \geq 1$. If we set $f_r=\min\left (\strut n_0,m_r(g-2)-1\right)$,
the sequence $u_n-1$ is increasing from the term of index $1$ 
(which is equal to $(B_r-1)/q^r$) until the term of index $f_r$. Hence
$$\Delta_r \geq \frac{(B_r-1)}{q^r}f_r.$$
\end{proof}

\begin{remarque}
Note that when the genus $g$ is growing to infinity and when $B_r>q^r$, then 
$f_r=\left\lfloor \frac{B_r-q^r}{q^r-1}\right\rfloor+1$ for $g$ sufficiently large,
by the Drinfeld-Vladut bound. 
\end{remarque}

\begin{lemme}\label{minoqr1}
Let $q$ and $r$ be two integers $>0$ and let $g>3$ be an integer.
Let $Q_r$ be the sum defined by:
$$Q_r=\sum_{m=0}^{m_r(g-2)-1} \frac{1}{(q^{r})^m}
\left (
\begin{array}{c}
B_r+m-1\\
 B_r-1
\end{array}
\right )
,$$ 
where $m_r(n)$ denotes the quotient of the Euclidian division of $n$ by $r$.
We have the following lower bound
$$Q_r\geq 
\left\{
\begin{array}{ll}
\left(\frac{q^r}{q^r-1} \right)^{B_r-1} & \hbox{ if } B_r \leq m_r(g-2),\\
\left(1+\frac{B_r}{q^r(m_r(g-2)-1)} \right)^{m_r(g-2)-1} &
\hbox{ if } B_r > m_r(g-2).
\end{array}
\right .
$$ 
\end{lemme}

\begin{proof}
Let $N$ and $k$ be two integers such that $N \geq k \geq 0$.
Let us study the following sum:
\begin{equation}\label{sec1:eq1}
 Q(N,k,x)=\sum_{i=0}^{N-k}\left( \begin{array}{c} k+i\\ k\end{array}\right)x^i.
\end{equation}
By derivating $k$ times the classical equality
$$\sum_{i=0}^N x^i=\frac{1-x^{N+1}}{1-x},$$
the following formula holds:
\begin{equation}\label{sec1:eq2}
Q(N,k,x)=\frac{1}{(1-x)^{k+1}}-\frac{x^{N-k+1}}{1-x}
\sum_{i=0}^{k}\left( \begin{array}{c} N+1\\ k-i\end{array}\right)\left(\frac{x}{1-x} \right)^i.
\end{equation}
Let us set $j=N+1-k+i$ in the previous formula,
then 
$$Q(N,k,x)=\frac{1}{(1-x)^{k+1}}-(1-x)^{N-k}
\sum_{j=N+1-k}^{N+1}\left( \begin{array}{c} N+1\\ j\end{array}\right)\left(\frac{x}{1-x} \right)^{j}. $$
Then
$$Q(N,k,x)=\frac{1}{(1-x)^{k+1}}-(1-x)^{N-k}
\sum_{j=0}^{N+1}\left( \begin{array}{c} N+1\\ j\end{array}\right)\left(\frac{x}{1-x} \right)^{j}+$$
$$
(1-x)^{N-k}\sum_{j=0}^{N-k}\left( \begin{array}{c} N+1\\ j\end{array}\right)\left(\frac{x}{1-x} \right)^{j},$$
$$Q(N,k,x)=\frac{1}{(1-x)^{k+1}}-(1-x)^{N-k}\left(1+\frac{x}{1-x} \right)^{N+1}+$$
$$(1-x)^{N-k}\sum_{j=0}^{N-k}\left( \begin{array}{c} N+1\\ j\end{array}\right)\left(\frac{x}{1-x} \right)^{j},$$
\begin{equation}\label{sec1:eq3}
Q(N,k,x)=(1-x)^{N-k} \sum_{j=0}^{N-k}\left( \begin{array}{c} N+1\\ j\end{array}\right)
\left(\frac{x}{1-x} \right)^j.
\end{equation}
%
%
Let us remark that $Q_r=Q\left(\strut m_r(g-2)+B_r-2,B_r-1,1/q^r\right)$. From
Formula (\ref{sec1:eq3}) we obtain
$$Q_r=\left(\frac{q^r-1}{q^r}\right)^{m_r(g-2)-1}\left( \sum_{j=0}^{m_r(g-2)-1}
\left( \begin{array}{c} m_r(g-2)+B_r-1\\ j\end{array}\right)
\left(\frac{1}{q^r-1} \right)^j\right).$$
If $B_r \leq m_r(g-2)$ then $m_r(g-2) \geq \frac{1}{2}\left(\strut m_r(g-2)+B_r \right)$.
Let us set $$s= \left\lfloor \frac{1}{2}\left(\strut m_r(g-2)+B_r-1 \right)\right\rfloor.$$
Let us remark that
$$\sum_{j=0}^s 
\left( \begin{array}{c} m_r(g-2)+B_r-1\\ j\end{array}\right)
\left(\frac{1}{q^r-1} \right)^j+$$
$$\left(\frac{1}{q^r-1} \right)\left(\sum_{j=s+1}^{m_r(g-2)+B_r-1}
\left( \begin{array}{c} m_r(g-2)+B_r-1\\ j\end{array}\right)
\left(\frac{1}{q^r-1} \right)^{j-1}\right)= $$
$$\left(\frac{q^r}{q^r-1}\right)^{m_r(g-2)+B_r-1}.$$
But
$$ \sum_{j=0}^s 
\left( \begin{array}{c} m_r(g-2)+B_r-1\\ j\end{array}\right)
\left(\frac{1}{q^r-1} \right)^j \geq$$
$$\sum_{s+1}^{m_r(g-2)+B_r-1}
\left( \begin{array}{c} m_r(g-2)+B_r-1\\ j\end{array}\right)
\left(\frac{1}{q^r-1} \right)^{j-1}.$$
Then
$$\sum_{j=0}^s 
\left( \begin{array}{c} m_r(g-2)+B_r-1\\ j\end{array}\right)
\left(\frac{1}{q^r-1} \right)^j \geq$$
$$ \frac{q^r-1}{q^r}\left(\frac{q^r}{q^r-1}\right)^{m_r(g-2)+B_r-1}.$$
Therefore
$$Q_r \geq  \left(\frac{q^r-1}{q^r}\right)^{m_r(g-2)-1}\frac{q^r-1}{q^r}
\left(\frac{q^r}{q^r-1}\right)^{m_r(g-2)+B_r-1},$$
$$Q_r \geq \left(\frac{q^r}{q^r-1} \right)^{B_r-1}. $$
If $B_r > m_r(g-2)$, then 

$$\frac{ \left( \begin{array}{c} m_r(g-2)+B_r-1\\ j\end{array}\right)}
{\left( \begin{array}{c} m_r(g-2)-1\\ j\end{array}\right)}=
\prod_{i=0}^{j-1}\frac{m_r(g-2)+B_r-1-i}{m_r(g-2)-1-i}\geq$$
$$
\left(\frac{m_r(g-2)+B_r-1}{m_r(g-2)-1}\right)^j.$$

Then
$$\sum_{j=0}^{m_r(g-2)-1}
\left( \begin{array}{c} m_r(g-2)+B_r-1\\ j\end{array}\right)
\left(\frac{1}{q^r-1} \right)^j\geq$$
$$
\sum_{j=0}^{m_r(g-2)-1}
\left( \begin{array}{c} m_r(g-2)-1\\ j\end{array}\right)
\left[\left(\frac{1}{q^r-1} \right)\left(\frac{m_r(g-2)+B_r-1}{m_r(g-2)-1}\right)\right]^j \geq$$
$$\left(1+\frac{m_r(g-2)+B_r-1}{(q^r-1)(m_r(g-2)-1)}\right)^{m_r(g-2)-1}.$$
Therefore
$$Q_r \geq \left(\frac{q^r-1}{q^r}\right)^{m_r(g-2)-1} 
\left(1+\frac{m_r(g-2)+B_r-1}{(q^r-1)(m_r(g-2)-1)}\right)^{m_r(g-2)-1}\geq$$
$$\left(1+\frac{B_r}{q^r(m_r(g-2)-1)} \right)^{m_r(g-2)-1}.
$$
\end{proof} 

\begin{lemme}\label{minoqr2}
Let $q$ and $r$ be two integers $>0$ and let $g>3$ be an integer.
Let $Q_r$ be the sum defined by:
$$Q_r=\sum_{m=0}^{m_r(g-2)-1} \frac{1}{(q^{r})^m}
\left (
\begin{array}{c}
B_r+m-1\\
 B_r-1
\end{array}
\right )
,$$ 
where $m_r(n)$ denotes the quotient of the Euclidian division of $n$ by $r$.
If $B_r+1 \leq (m_r(g-2)-1)(q^r-1)$ we get the following lower bound
$$Q_r\geq 
\left(\frac{q^r}{q^r-1} \right)^{B_r} - B_r \left (\begin{array}{c}
B_r+m_r(g-2)-1\\
 B_r
\end{array}
\right )\left(\frac{1}{q^r}\right)^{m_r(g-2)}.
$$
\end{lemme}
\begin{proof}
The term $\left(\frac{q^r}{q^r-1} \right)^{B_r}$ is the value of the infinite sum
$$\sum_{m=0}^{\infty} \frac{1}{(q^{r})^m}
\left (
\begin{array}{c}
B_r+m-1\\
 B_r-1
\end{array}
\right ).
$$
As proved later in Lemma \ref{remindermaj} the term 
$$B_r \left (\begin{array}{c}
B_r+m_r(g-2)-1\\
 B_r
\end{array}
\right )\left(\frac{1}{q^r}\right)^{m_r(g-2)}$$
is an upper bound of the remainder.
\end{proof}

\section{Lower bounds on the class number}\label{sec:class}
In this section, using the lower bounds obtained in the previous section, 
we derive lower bounds on the class number $h=h(F/{\mathbb{F}}_q)$ of an algebraic function field
$F/{\mathbb{F}}_q$.
Let $B_r=B_r(F/{\mathbb{F}}_q)$ be the number of places of degree $r$.
\begin{theoreme}\label{finite} 
Let $F/{\mathbb{F}}_q$ be an algebraic function field defined over ${\mathbb{F}}_q$ of genus $g\geq 2$ and 
$h$ the class number of $F/{\mathbb{F}}_q$. Suppose that the numbers $B_1$ and $B_r$ of places of 
degree respectively $1$ and $r$
are such that $B_1>0$ and $B_r>0$. Let us denote by 
$K_1(i,B)$ and $K_2(q,j,B)$ 
the following numbers:
$$K_1(i,B)= 
\left (
\begin{array}{c}
 B+i\\
 B
\end{array}
\right )
\quad \hbox{and}\quad
K_2(q,j,B)= \sum_{i=0}^{j}\frac{1}{q^{i}}
\left (
\begin{array}{c}
B+i-1\\
 B-1
\end{array}
\right )
.$$
Then, the following inequality holds:
\begin{equation}\label{minh}
h \geq \frac{(q-1)^2 K_2(q,r-1,B_1)}{(g+1)(q+1)-B_1}  q^{g-1} 
Q_r  +
\end{equation}
$$
\frac{q(q-1)^2K_2\left(\strut q,s_r(g-2),B_1\right)}{(g+1)(q+1)-B_1}\left (
\begin{array}{c}
 B_r+m_r(g-2)-1\\
 B_r-1
\end{array}
\right )+
$$
$$ 
\frac{(q-1)^2K_1(r-1,B_1)}{ \strut(g+1)(q+1)-B_1}
\left (
\begin{array}{c}
 B_r+m_r(g-1)-1\\
 B_r
\end{array}
\right )+
$$
$$
\frac{(q-1)^2K_1\left(\strut s_r(g-1),B_1\right )}{ \strut(g+1)(q+1)-B_1}
\left (
\begin{array}{c}
 B_r+m_r(g-1)-1\\
 B_r-1
\end{array}
\right )
$$
where $Q_r$ can be bounded by any of the three lemmas \ref{minoqr}, \ref{minoqr1} and \ref{minoqr2}.
\end{theoreme} 

\begin{proof}
The result is a direct consequence of the equality (\ref{mainformula}) 
which gives an expression of $h$, the upper bound (\ref{R2}), the lower bound (\ref{majsom1r})
and Lemmas \ref{minoqr}, \ref{minoqr1} and \ref{minoqr2}. 
\end{proof}

Let us give some simplified formulas which are slightly less accurate but more readable.

\begin{corollaire}
Let $F/{\mathbb{F}}_q$ be an algebraic function field defined over ${\mathbb{F}}_q$ of genus $g\geq 2$ and 
$h$ the class number of $F/{\mathbb{F}}_q$. Suppose that the numbers $B_1$ and $B_r$ of places of 
degree respectively $1$ and $r$
are such that $B_1>0$ and $B_r>0$. Then the following four results hold:
\begin{enumerate}
 \item Let us set 
$$f_r=
\left\{
\begin{array}{lcl}
0 & \hbox{ if } & \frac{g-1}{2}<r\leq g-1;\\
1 & \hbox{ if } & r\leq\frac{g-1}{2} \hbox{ and } B_r<q^r;\\
\min\left(\lfloor \frac{B_r-q^r}{q^r-1} \rfloor+1,m_r(g-2)-1\right) & \hbox{ if } 
& r\leq \frac{g-1}{2} \hbox{ and } B_r\geq q^r;
\end{array}
\right.
$$
then  $$h \geq \frac{(q-1)^2}{(g+1)(q+1)}q^{g-1}\left(\frac{q^{rm_r(g-2)}-1}{q^{r(m_r(g-2)-1)}(q^r-1)}+
\frac{(B_r-1)}{q^r}f_r\right);$$
 \item if $B_r\leq m_r(g-2)$ then
$$h \geq \frac{(q-1)^2}{(g+1)(q+1)}q^{g-1}\left(\frac{q^r}{q^r-1} \right)^{B_r-1};$$
 \item if $B_r > m_r(g-2)$ then
$$h \geq \frac{(q-1)^2}{(g+1)(q+1)}q^{g-1}\left(1+\frac{B_r}{q^r(m_r(g-2)-1)} \right)^{m_r(g-2)-1};$$
 \item if $B_r+1 \leq (m_r(g-2)-1)(q^r-1)$ then
$$h \geq \frac{(q-1)^2}{(g+1)(q+1)}q^{g-1}\left[\left(\frac{q^r}{q^r-1} \right)^{B_r} - B_r \left (\begin{array}{c}
B_r+m_r(g-2)-1\\
 B_r
\end{array}
\right )\left(\frac{1}{q^r}\right)^{m_r(g-2)}\right].$$
\end{enumerate}

\end{corollaire}

\begin{remarque} 
If we set $L_1=q^{g-1}\frac{(q-1)^2}{(q+1)(g+1)-B_1}$ which corresponds, to the first lower bound on $h$ 
given in \cite[Théorème 2]{lamd} and if we denote by $L'_r$ the bound  (\ref{minh}), then
$$
L'_r = K_2(q,r-1,B_1)\, Q_r
L_1+
$$
$$
\frac{q(q-1)^2K_2\left(\strut q,s_r(g-2),B_1\right)}{(g+1)(q+1)-B_1}\left (
\begin{array}{c}
 B_r+m_r(g-2)-1\\
 B_r-1
\end{array}
\right )+
$$
$$
\frac{(q-1)^2K_1(r-1,B_1)}{\strut(g+1)(q+1)-B_1}
\left (
\begin{array}{c}
 B_r+m_r(g-1)-1\\
 B_r
\end{array}
\right )
+$$
$$
\frac{(q-1)^2K_1(s_r(g-1),B_1)}{\strut(g+1)(q+1)-B_1}
\left (
\begin{array}{c}
 B_r+m_r(g-1)-1\\
 B_r-1
\end{array}
\right ),
$$
which gives in the case $r=1$ and if we choose the lower bound on $Q_r$ given by Lemma \ref{minoqr}:
$$
L'_1 = \left(\frac{q^{g-2}-1}{q^{g-3}(q-1)}+\frac{(B_1-1)f_1}{q}\right) L_1+
$$
$$
\frac{q(q-1)^2}{(g+1)(q+1)-B_1}\left (
\begin{array}{c}
 B_1+g-3\\
 B_1-1
\end{array}
\right )+
$$

$$
\frac{(q-1)^2}{\strut(g+1)(q+1)-B_1}
\left (
\begin{array}{c}
 B_1+g-1\\
 B_1
\end{array}
\right )
.$$
\end{remarque}

\begin{remarque}
 Let us remark that in \cite[Theorem 2]{lamd} Lachaud - Martin-Deschamps study the worst case
(see also Theorem \ref{theolamd}).
In the worst case $r=1$, $B_1=1$ and $g\geq 2$ (third bound given in \cite[Theorem 2]{lamd})
we obtain the same lower bound. More precisely, if we denote by $L_3$ the third bound 
in \cite[Theorem 2]{lamd}), namely
$$L_3=(q^g-1)\frac{q-1}{q+g+qg},$$
then our bound $L'_1$ obtained  using the lower bound of $Q_r$
given by Lemma \ref{minoqr} satisfies the following inequality:
$$L'_1 \geq L_3+ \frac{(q-1)^2q^{g-1}}{q+g+gq}\,\frac{B_1-1}{q}\,f_1. $$
\end{remarque}

\section{Asymptotical bounds with respect to the genus $g$}\label{sec:asymptotic}

We now give asymptotical bounds on the class number of certain sequences of algebraic function fields 
defined over a finite field when $g$ tends to infinity and when 
there exists an integer $r$ such that:
$$\liminf_{k \rightarrow +\infty} \frac{B_r(F_k/{\mathbb{F}}_{q})}{g(F_k)}>0.$$ 

\medskip

We consider a sequence of function fields $F_k/\F_q$ defined over the 
finite field $\F_q$.
Let us denote by $g_k$ the genus of $F_k$ and by $B_r(k)$ the number of places of degree
$r$ of $F_k$. We suppose that the sequence $g_k$ is growing to infinity and that for a certain $r$
the following holds:
$$\liminf_{k \rightarrow +\infty} \frac{B_r(k)}{g_k}=\mu_r >0.$$ 
We also assume that $B_1(k) \geq 1$ for any $k$.

Let us denote by $m_r(x)$ the Euclidean quotient of $x$ by $r$.
Let $a$ and $b$ be two constant integers and
let $\eta$ be any given number such that $0 < \eta <1 $. Then there is an integer 
$k_0>0$ such that for any integer $k\geq k_0$ the following inequalities hold:

\begin{equation}\label{meq1}
\mu_r g_k(1-\eta) < B_r(k) < 
\frac{g_k}{r}\left(q^{\frac{r}{2}}-1 \right)(1+\eta).
\end{equation}

\begin{equation}\label{meq2}
\frac{g_k}{r}(1-\eta) <m_r(g_k-a)-b < \frac{g_k}{r}(1+\eta).
\end{equation}

We study the asymptotic behaviour of the binomial coefficient
$$\left( 
\begin{array}{c}
B_r(k)+m_r(g_k-a)-b\\
B_r(k)
\end{array}
\right),$$

and for this purpose we do the following:
\begin{enumerate}
 \item we use  the generalized binomial coefficients (defined 
for example with the Gamma function);
 \item we remark that 
if $u$ and $v$ are positive real numbers and if $u'\geq u$ and $v' \geq v$,
then the following inequality between generalized binomial coefficients holds:
$$\left( 
\begin{array}{c}
u'+v'\\
u'
\end{array}
\right) \geq \left( 
\begin{array}{c}
u+v\\
u
\end{array}
\right);$$
 \item we use the Stirling formula.
\end{enumerate}

So, we obtain:

\begin{lemme}\label{mainlemma}
For any real numbers  $\epsilon,\eta$ such that $\epsilon>0$ and $0<\eta<1$,
there exists an integer $k_1$ such that for any integer $k\geq k_1$ the following holds:
\begin{equation}\label{dineq}
(1-\epsilon)\frac{C_1}{\sqrt{g_k}} q_1^{g_k} <  \left( 
\begin{array}{c}
B_r(k)+m_r(g_k-a)-b\\
B_r(k)
\end{array}
\right) < (1+\epsilon)\frac{C_2}{\sqrt{g_k}} q_2^{g_k},
\end{equation}
where 
$$C_1=\sqrt{\frac{r\mu_r+1}{2\pi r\mu_r(1-\eta)}}$$ 
and 
$$C_2=\sqrt{\frac{q^{\frac{r}{2}}}{2\pi(q^{\frac{r}{2}}-1)(1+\eta)}}$$
are two positive constants depending upon $\eta$ and where
$$q_1=\left(\frac{\left(\mu_r+\frac{1}{r} \right)^{\mu_r+\frac{1}{r}}}
{\mu_r^{\mu_r}\left(\frac{1}{r} \right)^{\frac{1}{r}}} \right)^{1 -\eta}$$
and
$$q_2=\left(\left(\frac{q^{\frac{r}{2}}}{q^{\frac{r}{2}}-1}\right)^{\frac{1}{r}(q^{\frac{r}{2}}-1)}\sqrt{q}  
\right)^{1 +\eta}.$$
\end{lemme}

\begin{proof}
 By the previous remark we know that:
$$
\left(
\begin{array}{c}
g_k(1-\eta) \left(\mu_r+\frac{1}{r} \right)\\
g_k(1-\eta) \mu_r
\end{array}
\right) \leq 
\left( 
\begin{array}{c}
B_r(k)+m_r(g_k-a)-b\\
B_r(k)
\end{array}
\right) \leq$$
$$
\left(
\begin{array}{c}
\frac{g_k}{r}q^{\frac{r}{2}}(1+\eta) \\
\frac{g_k}{r}\left(q^{\frac{r}{2}}-1 \right)(1+\eta) 
\end{array}
\right).
$$
Using the Stirling formula we get
$$
\left(
\begin{array}{c}
g_k(1-\eta) \left(\mu_r+\frac{1}{r} \right)\\
g_k(1-\eta) \mu_r
\end{array}
\right)\underset{k\rightarrow +\infty}{\sim} \frac{C_1}{\sqrt{g_k}} 
\left(\left(\frac{\left(\mu_r+\frac{1}{r} \right)^{\mu_r+\frac{1}{r}}}
{\mu_r^{\mu_r}\left(\frac{1}{r} \right)^{\frac{1}{r}}} \right)^{1 -\eta}\right)^{g_k}
$$
and

$$ 
\left(
\begin{array}{c}
\frac{g_k}{r}q^{\frac{r}{2}}(1+\eta) \\
\frac{g_k}{r}\left(q^{\frac{r}{2}}-1 \right)(1+\eta) 
\end{array}
\right) \underset{k\rightarrow +\infty}{\sim} 
\frac{C_2}{\sqrt{g_k}}\left(\left( \left(\frac{q^{\frac{r}{2}}}{q^{\frac{r}{2}}-1} \right)
^{\frac{1}{r}(q^{\frac{r}{2}}-1) } 
\sqrt{q}\right)^{1+\eta} \right )^{g_k}.
$$
\end{proof}

Let us recall that for evaluating asymptotically the class number $h_k$, 
we need to evaluate asymptotically the sums $\Sigma_1(k)$ and $\Sigma_2(k)$, 
respectively defined by the formulas (\ref{majsom1r}) and  (\ref{S2}), 
with respect to the parameter $k$.

\subsection{Asymptotical lower bound on $\Sigma_2(k)$}

First, let us consider the sum

$$\Sigma_2(k)=q^{g_k-1}\sum_{n=0}^{g_k-2}\frac{A_n}{q^n}
\geq q^{g_k-1} K_2\left (\strut q,r-1,B_1(k) \right ) Q_r(k),
$$
where 
$$K_2\left (\strut q,r-1,B_1(k)\right )=
\sum_{i=0}^{r-1}\frac{1}{q^i}
\left(
\begin{array}{c}
 B_1(k)+i-1\\
B_1(k)-1
\end{array}
\right)
$$
and
$$Q_r(k)=\sum_{n=0}^{m_r(g_k-2)-1}\frac{1}{q^{nr}}
\left(
\begin{array}{c}
B_r(k)+n-1\\
B_r(k)-1
\end{array}
\right).$$

For evaluating asymptotically the sum $\Sigma_2(k)$, we have to evaluate asymptotically $Q_r(k)$. 
Let us set 
$$T_r(X,k)=
\sum_{n=0}^{M_r(k)}
\left(
\begin{array}{c}
B_r(k)+n-1\\
B_r(k)-1
\end{array}
\right) X^n,
$$
where $M_r(k)=m_r(g_k-2)-1$, then
$$Q_r(k)= T_r\left(\frac{1}{q^r},k\right).$$
Let us set
$$S_r(X,k)=
\sum_{n=0}^{\infty}
\left(
\begin{array}{c}
B_r(k)+n-1\\
B_r(k)-1
\end{array}
\right) X^n,
$$
then
$$S_r(X,k)=\frac{1}{(1-X)^{B_r(k)}}.$$
But $S_r(X,k)=T_r(X,k)+R_r(X,k)$, where
$$R_r(X,k)=
\sum_{n>M_r(k)}
\left(
\begin{array}{c}
B_r(k)+n-1\\
B_r(k)-1
\end{array}
\right) X^n.
$$
By the Taylor formula the following holds:
$$ R_r(X,k)=
\int_0^X \frac{(X-t)^{M_r(k)}}{M_r(k)!} S_r^{(M_r(k)+1)}(t,k) dt.$$
We have to compute the successive derivatives of $S_r(X,k)$:
$$ S_r^{(M_r(k)+1)}(X,k)= \frac{B_r(k)\, M_r(k)!}{(1-X)^{B_r(k)+M_r(k)+1}} 
\left(
\begin{array}{c}
 B_r(k)+M_r(k)\\
B_r(k)
\end{array}
\right) .
$$
Then
$$ R_r(X,k)=
B_r(k)
\left(
\begin{array}{c}
 B_r(k)+M_r(k)\\
B_r(k)
\end{array}
\right)
\int_0^X \frac{(X-t)^{M_r(k)}}{(1-t)^{B_r(k)+M_r(k)+1}}dt
$$

\begin{lemme}\label{remindermaj}
There is an integer $k_2$ such that for $k \geq k_2$ the function
$$f_k(t)= \frac{(\frac{1}{q^r}-t)^{M_r(k)}}{(1-t)^{B_r(k)+M_r(k)+1}}$$
is decreasing on $[0,\frac{1}{q^r}]$.
\end{lemme}

\begin{proof}
 $$f'(t)=$$
 $$\frac{(\frac{1}{q^r}-t)^{M_r(k)-1}}{(1-t)^{M_r(k)+B_r(k)+2}} 
\left ( -M_r(k)(1-t)+\left (\strut M_r(k)+B_r(k)+1\right )\left(\frac{1}{q^r}-t\right)\right).$$
 The derivative vanishes for $t_1=\frac{1}{q^r}$ and for 
 $$t_0=-\frac{M_r(k)}{B_r(k)+1}\left(1-\frac{1}{q^r}\right)+\frac{1}{q^r}.$$
 Let us choose a real number $\epsilon$ such that $0<\epsilon <1$.
 From (\ref{meq1}) and (\ref{meq2}) there exists an integer $k_2$ such that for $k \geq k_2$ the following holds:
 $$\frac{M_r(k)}{B_r(k)+1}  \geq \frac{1-\epsilon}{q^{\frac{r}{2}}-1}.$$
 Hence
 $$t_0 \leq -\frac{(1-\epsilon)(q^r-1)}{q^{\frac{r}{2}}-1}\frac{1}{q^r}+\frac{1}{q^r}.$$
 If $\epsilon$ is choosen sufficiently small, then
 $$\frac{(1-\epsilon)(q^r-1)}{q^{\frac{r}{2}}-1}>1,$$
 and $t_0 <0$.
\end{proof}

We conclude that 
$$R_r\left(\frac{1}{q^r},k\right)\leq B_r(k)\left(
\begin{array}{c}
 B_r(k)+M_r(k)\\
B_r(k)
\end{array}
\right) \left(\frac{1}{q^r}\right)^{M_r(k)+1}.$$

\begin{lemme}\label{remainder1}
Let $q$ be a prime power and let $r\geq 1$ be an integer.
Then if $q\geq 4$ or $r\geq 2$, the sequence of remainder terms 
$R_r\left(\frac{1}{q^r},k\right)$ is such that 
$$\lim_{k\rightarrow +\infty}\frac{R_r\left(\frac{1}{q^r},k\right)}{S_r(\frac{1}{q^r},k)}=0$$
\end{lemme}

\begin{proof}
By Lemma (\ref{mainlemma}), for any real number $\epsilon$ and $\eta$ such that $\epsilon>0$ and $0<\eta<1$, 
we get $$R_r\left(\frac{1}{q^r},k\right) <(1+\epsilon) B_r(k) 
 \frac{C_2}{\sqrt{g_k}} q_2^{g_k} \left(\frac{1}{q^r}\right)^{M_r(k)+1}$$ 
where
$$q_2=\left(\left(\frac{q^{\frac{r}{2}}}{q^{\frac{r}{2}}-1}\right)^{\frac{1}{r}(q^{\frac{r}{2}}-1)}  
\sqrt{q}\right)^{1 +\eta} \hbox{ and }
C_2=\sqrt{\frac{q^{\frac{r}{2}}}{2\pi(q^{\frac{r}{2}}-1)(1+\eta)}}.$$
But
$$\ln \left(\left(\frac{q^{\frac{r}{2}}}{q^{\frac{r}{2}}-1}\right)^{\frac{1}{r}(q^{\frac{r}{2}}-1)}  
\right)={\frac{1}{r}(q^{\frac{r}{2}}-1)}\ln\left(1+\frac{1}{q^{\frac{r}{2}}-1} \right)<\frac{1}{r}.$$
Then
\begin{equation}\label{durdur}
\left(\frac{q^{\frac{r}{2}}}{q^{\frac{r}{2}}-1}\right)^{\frac{1}{r}(q^{\frac{r}{2}}-1)}  
< e^{\frac{1}{r}}.
\end{equation}
A direct computation for the two particular cases $q=2, r=2$ and $q=4,r=1$ shows that
for these values we have 
$$\left(\frac{q^{\frac{r}{2}}}{q^{\frac{r}{2}}-1}\right)^{\frac{1}{r}(q^{\frac{r}{2}}-1)}=\sqrt{q}.$$
In all the other cases we derive from the inequality (\ref{durdur}) 

$$\left(\frac{q^{\frac{r}{2}}}{q^{\frac{r}{2}}-1}\right)^{\frac{1}{r}(q^{\frac{r}{2}}-1)}<\sqrt{q}.$$
Hence, if $\left(\frac{q^{\frac{r}{2}}}{q^{\frac{r}{2}}-1}\right)^{\frac{1}{r}(q^{\frac{r}{2}}-1)}<\sqrt{q}$ then for $\eta$ sufficiently small, $q_2<q$ and so 
$$\lim_{k\rightarrow +\infty} B_r(k) q_2^{g_k} \left(\frac{1}{q^r}\right)^{M_r(k)+1}=0.$$
Moreover, $$S_r\left(\frac{1}{q^r},k\right)=\left(\frac{q^r}{q^r-1}\right)^{B_r(k)}$$ which tends to the infinity.

Now if $\left(\frac{q^{\frac{r}{2}}}{q^{\frac{r}{2}}-1}\right)^{\frac{1}{r}(q^{\frac{r}{2}}-1)}=\sqrt{q}$ 
(i.e if $q=2$ and $r=2$ or if $q=4$ and $r=1$) we have 
$q_2=q^{1+\eta}$. But 
$B_r(k)\geq \alpha g_k$ 
where $\alpha$ is a positive constant.
Then for $\eta$ sufficiently small, we have 
$$1<\frac{q_2}{q}=q^{\eta} <\left(\frac{q^r}{q^r-1}\right)^{\alpha},$$ 
which implies 
$$\left(\frac{q_2}{q}\right)^{g_k}\in o\left(S_r\left(\frac{1}{q^r},k\right)\right).$$
\end{proof}




Now, we can establish the following proposition:

\begin{proposition}\label{somme2}
Let $q$ be a prime power $q$ and let $r\geq 1$ be an integer. Assume that 
$$\liminf_{k \rightarrow +\infty} \frac{B_r(k)}{g_k}=\mu_r >0.$$
Then, the sum 
$$Q_r(k)=\sum_{n=0}^{m_r(g_k-2)-1}\frac{1}{q^{nr}}
\left(
\begin{array}{c}
B_r(k)+n-1\\
B_r(k)-1
\end{array}
\right)$$ 
satisfies $$Q_r(k)\underset{k\rightarrow +\infty}{\sim}\left(\frac{q^r}{q^r-1}\right)^{B_r(k)}.$$
Hence, the sequence $$\Sigma_2(k)=q^{g_k-1}\sum_{n=0}^{g_k-2}\frac{A_n}{q^n}$$ satisfies
$$\Sigma_2(k)\in\Omega\left(\left(\frac{q^r}{q^r-1}\right)^{B_r(k)}q^{g_k}\right).$$
\end{proposition}

\begin{proof}
If $q\geq 4$ or $r\geq 2$, the result is given by Lemma \ref{remainder1}.
Let us consider the particular cases of $r=1$ and $q=2$ or $q=3$. In these cases, 
first assume that the following limit exists:$$\mu_1=\lim_{k \rightarrow +\infty} \frac{B_1(k)}{g_k}>0.$$
Let us do the same study as in Lemma \ref{mainlemma}.
Let $a$ be a constant integer and
let $\eta$ be any given number such that $0 < \eta <1 $. Then there is an integer 
$k_0>0$ such that for any integer $k\geq k_0$ the following inequalities hold:

\begin{equation}\label{meq11}
\mu_1 g_k(1-\eta) < B_1(k) < 
\mu_1 g_k(1+\eta).
\end{equation}

We study the asymptotic behaviour of the binomial coefficient:
$$\left( 
\begin{array}{c}
B_1(k)+g_k-a\\
B_1(k)
\end{array}
\right).$$

For any real numbers  $\epsilon,\eta$ such that $\epsilon>0$ and $0<\eta<1$,
there exists an integer $k_1$ such that for any integer $k\geq k_1$ the following holds:
\begin{equation}\label{dineq2}
(1-\epsilon)\frac{C_1}{\sqrt{g_k}} q_1^{g_k} <  \left( 
\begin{array}{c}
B_1(k)+g_k-a \\
B_1(k)
\end{array}
\right) < (1+\epsilon)\frac{C_2}{\sqrt{g_k}} q_2^{g_k},
\end{equation}
where 
$$C_1=\sqrt{\frac{\mu_1+1}{2\pi \mu_1(1-\eta)}}$$ 
and 
$$C_2=\sqrt{\frac{\mu_1+1}{2\pi \mu_1(1+\eta)}}$$
are two positive constants depending upon $\eta$ and where
$$q_1=\left(\frac{\left(\mu_1+1 \right)^{\mu_1+1}}
{\mu_1^{\mu_1}}\right)^{1 -\eta}$$
and
$$q_2=\left(\frac{\left(\mu_1+1 \right)^{\mu_1+1}}
{\mu_1^{\mu_1}}\right)^{1 +\eta}.$$



Hence, 

\begin{equation}\label{limRS}
\lim_{k\rightarrow +\infty}\frac{R_1\left(\frac{1}{q},k\right)}{S_1(\frac{1}{q},k)}=0.
\end{equation}


Indeed, for any real number $\epsilon$ and $\eta$ such that $\epsilon>0$ and $0<\eta<1$, 
we get 
$$R_1\left(\frac{1}{2},k\right) <C_2 q^2 (1+\epsilon) (1+\eta)\mu_1 \sqrt{g_k}
  \left(\frac{\left(\mu_1+1 \right)^{\mu_1+1}}
{\mu_1^{\mu_1}}\right)^{g_k(1 +\eta)}\left(\frac{1}{q}\right)^{g_k}$$
and 
$$S_1(\frac{1}{q},k)\geq \left(\frac{q}{q-1}\right)^{\mu_1 g_k(1-\eta)}.$$
Then
$$\frac{R_1\left(\frac{1}{2},k\right)}{S_1(\frac{1}{q},k)} 
\leq  $$ $$C_2 q^2 (1+\epsilon) (1+\eta)\mu_1 \sqrt{g_k}\left[
  \left(\frac{\left(\mu_1+1 \right)^{\mu_1+1}}
{\mu_1^{\mu_1}}\right)^{(1 +\eta)}\left(\frac{q-1}{q} \right)^{\mu_1(1-\eta)}
\left(\frac{1}{q}\right)\right]^{g_k}.$$
Let us study for $0<\mu_1 \leq \sqrt{q}-1$ the function
$$f(\mu_1)=\left(\frac{\left(\mu_1+1 \right)^{\mu_1+1}}
{\mu_1^{\mu_1}}\right)\left(\frac{q-1}{q}\right)^{\mu_1}\frac{1}{q}.$$
This function is increasing since the derivative
$$(\log(f(\mu_1))'=\frac{f'(\mu_1)}{f(\mu_1)}=\log\left(\left(1+\frac{1}{\mu_1}\right)\left(1-\frac{1}{q}\right)\right)>0$$
with $\mu_1 \leq \sqrt{q}-1$. 
In particular, we have:

\begin{enumerate}
 \item for $q=2$, $f(\mu_1) <0.89$;
 \item for $q=3$, $f(\mu_1) <0.81$.
\end{enumerate}
Then in each case, for $\eta$ sufficiently small, 
\begin{equation}\label{equagene}
 \left[
  \left(\frac{\left(\mu_1+1 \right)^{\mu_1+1}}
{\mu_1^{\mu_1}}\right)^{(1 +\eta)}\left(\frac{q-1}{q} \right)^{\mu_(1-\eta)}\left(\frac{1}{q}\right)\right]<1,
\end{equation}
and consequently we obtain the limit (\ref{limRS}).
Then as
$$\lim_{k \rightarrow +\infty} \frac{B_1(k)}{g_k}=\mu_1 >0,$$
for any $\epsilon >0$ there exists $k_0$ such that for each $k\geq k_0$
$$Q_1(k)\geq (1-\epsilon)\left( \frac{q}{q-1}\right)^{B_1(k)}.$$
Suppose now that we have the weaker assumption
$$\liminf_{k \rightarrow +\infty} \frac{B_1(k)}{g_k}=\mu_1 >0$$
and suppose that there is a real number $\epsilon>0$ such that for any $k_0$,
there exists an integer $k\geq k_0$ such that 
$$Q_1(k)< (1-\epsilon)\left( \frac{q}{q-1}\right)^{B_1(k)}.$$
We can construct an infinite subsequence $(Q_1(k_i))_i$  of $(Q_1(k))_k$such that the previous
inequality holds for any $i$. Next we can extract from $k_i$ a subsequence $k_{i_j}$
such that $$\lim_{j \rightarrow +\infty} \frac{B_1(k_{i_j})}{g_k}=\mu\geq \mu_1 >0.$$
Using the previous result we conclude that for $j$ sufficiently large
$$Q_1(k_{i_j})\geq (1-\epsilon)\left( \frac{q}{q-1}\right)^{B_1(k)},$$
which gives a contradiction.


\end{proof}

\begin{corollaire}
Let $q$ be a prime power $q$ and let $r\geq 1$ be an integer. Assume that 

$$\liminf_{k \rightarrow +\infty} \frac{B_r(k)}{g_k}=\mu_r >0.$$
 For any $\alpha$ such that $0< \alpha < \mu_r$ the following holds:
$$\Sigma_2(k)\in\Omega\left(\left[\left(\frac{q^r}{q^r-1}\right)^{\alpha}q\right]^{g_k}\right).$$
\end{corollaire}



\subsection{Asymptotical lower bound on $\Sigma_1(k)$}

Let us consider now the sum
$$\Sigma_1(k)
=\sum_{n=0}^{g_k-1}A_n.$$ Using the inequality (\ref{majsom1r}), we will obtain
for this sum a lower bound which is negligible compared with 
the one obtained for $\Sigma_2(k)$. More precisely,
the best we can do with this inequality is given by the following proposition:

\begin{proposition}\label{somme1}
For any $\eta$ such that $0<\eta<1$, the sequence $\Sigma_1(k)$ satisfies
$$\Sigma_1(k) \in\Omega\left(\frac{q_1^{g_k}}{\sqrt{g_k}}\right)$$
where 
$$q_1=\left(\frac{\left(\mu_r+\frac{1}{r} \right)^{\mu_r+\frac{1}{r}}}
{\mu_r^{\mu_r}\left(\frac{1}{r} \right)^{\frac{1}{r}}} \right)^{1 -\eta}.$$
\end{proposition}

\begin{proof}
From the inequality (\ref{majsom1r}) we know that
$$ 
\Sigma_1(k)
\geq K_1(r,B_1(k)) 
\left (
\begin{array}{c}
 B_r(k)+m_r(g_k-1)\\
 B_r(k)
\end{array}
\right )+
$$

$$
K_1\left(\strut s_r(g-1),B_1(k)\right) 
\left (
\begin{array}{c}
 B_r(k)+m_r(g_k-1)\\
 B_r(k)-1
\end{array}
\right ).
$$
This inequality can be written in the following way
$$
\Sigma_1(k)
\geq 
$$

$$\left(K_1(r,B_1(k))+\frac{B_r(k)}{m_r(g_k-1)} K_1\left(\strut s_r(g-1),B_1(k)\right) \right)
\left (
\begin{array}{c}
 B_r(k)+m_r(g_k-1)\\
 B_r(k)
\end{array}
\right ).
$$
As for $k$ sufficiently large the following holds:
$$r\mu_r\left( \frac{1-\eta}{1+\eta}\right)\leq \frac{B_r(k)}{m_r(g_k-1)}\leq (q^{\frac{r}{2}}-1)
\left( \frac{1+\eta}{1-\eta}\right).$$
Then, the best we can obtain is
$$ 
\Sigma_1(k)=
\Omega\left(
\left (
\begin{array}{c}
 B_r(k)+m_r(g_k-1)\\
 B_r(k)
\end{array}
\right )
\right).
$$
Using Lemma \ref{mainlemma} we obtain conclude that
$$\Sigma_1(k) \in \Omega\left(\frac{q_1^{g_k}}{\sqrt{g_k}}\right),$$
where
$$q_1=\left(\frac{\left(\mu_r+\frac{1}{r} \right)^{\mu_r+\frac{1}{r}}}
{\mu_r^{\mu_r}\left(\frac{1}{r} \right)^{\frac{1}{r}}} \right)^{1 -\eta}.$$

\end{proof}

\begin{proposition}
The value $q_1$ introduced in the previous proposition is such that
 \begin{itemize}
  \item if $q \geq 4$ or $r \geq 2$ 
then $q_1<q$;
  \item if $r=1$ and $q=2$ or $q=3$ then for $\eta$ sufficiently small
$$q_1 < \left(\frac{q^r}{q^r-1}\right)^{\mu_1(1-\eta)}q.$$
\end{itemize}
\end{proposition}
\begin{proof}
 The function 
$$v(x)= \left(\frac{\left(x+\frac{1}{r} \right)^{x+\frac{1}{r}}}
{x^{x}\left(\frac{1}{r} \right)^{\frac{1}{r}}} \right)$$
is an increasing function. Then
$$v(\mu_r) \leq v\left(\frac{1}{r}\left(q^{\frac{r}{2}}-1\right) \right),$$
namely,
$$q_1 \leq \left(\left(\frac{q^{\frac{r}{2}}}{q^{\frac{r}{2}}-1}\right)^{\frac{1}{r}(q^{\frac{r}{2}}-1)}  
\sqrt{q}\right)^{1 -\eta}.$$
We refer to the proof of Lemma \ref{remainder1} to see that when $q \geq 4$ or $r\geq 2$,
the right member of
this last formula is $\leq q^{1-\eta}$.
When $r=1$ and $q=2$ or $q=3$ the formula (\ref{equagene}) obtained
in the proof of Proposition \ref{somme2} holds and gives us the result.
\end{proof}

\begin{conclusion}
We conclude that in any case, the lower bound found on $\Sigma_2$ is better than
the lower bound found on $\Sigma_1$.
\end{conclusion}

\subsection{Asymptotical lower bound on the class number}

In this section, we obtain asymptotical lower bounds on the class number of certain sequences of algebraic function fields
as the consequence of the conclusion of the previous section and of Proposition \ref{somme2}.

\begin{theoreme}\label{mainun}
Let ${\mathcal F}/\F_q=(F_k/\F_q)_k$ be a sequence of function fields over a finite field $\F_q$. 
Let us denote by $g_k$ the genus of
$F_k$, by $h(F_k/\F_q)$ the class number of $F_k/\F_q$ and by $B_i(F_k/\F_q)$ the number of places of degree $i$ of $F_k/\F_q$.  
Let us suppose that for any $k$ we have $B_1(F_k/\F_q)\geq 1$ and that there is an integer $r \geq 1$ such that
$$\liminf_{k\rightarrow\infty} \frac{B_r(F_k/\F_q)}{g_k}= \mu_r({\mathcal F}/\F_q) >0.$$
Then
for any $\alpha$ such that $0<\alpha<\mu_r({\mathcal F}/\F_q)$
$$h(F_k/\F_q) \in\Omega\left(\left({\left(\frac{q^r}{q^{r}-1}\right)^{\alpha}q}\right)^{g_k}\right).$$
\end{theoreme}

Moreover, we can improve the previous theorem in the following way:

\begin{theoreme}\label{maindeux}
Let ${\mathcal F}/\F_q=(F_k/\F_q)_k$ be a sequence of function fields over a finite field $\F_q$ and 
let $r\geq 1$ an integer.
Let ${\mathcal G}/\F_{q^r}=(G_k/\F_{q^r})_k$ be the degree $r$ 
constant field extension sequence of the sequence ${\mathcal F}/\F_q$, namely
for any $k$ we have $G_k/\F_{q^r}=F_k \otimes_{\F_q} \F_{q^r}$. 
Let us denote by $g_k$ the genus of
$F_k$ and $G_k$, by $h(F_k/\F_q)$ the class number of $F_k/\F_q$ and by $B_r(F_k/\F_{q})$ the number of places of degree $r$ of $F_k/\F_q$.  
Let us suppose that 
$$\liminf_{k\rightarrow\infty} \sum_{i\mid r}\frac{iB_i(F_k/\F_{q})}{g_k}= \mu_1({\mathcal G}/\F_{q^r}) >0.$$
Then for any $\alpha$ such that $0<\alpha<\frac{1}{r}\mu_1({\mathcal G}/\F_{q^r})$
$$h(F_k/\F_q) \in\Omega\left(\left({\left(\frac{q^r}{q^{r}-1}\right)^{\alpha}q}\right)^{g_k}\right).$$
\end{theoreme}

\begin{proof}
For $r=1$ the result is yet proved by the previous study, then we suppose $r>1$.
We know that
$$B_1(G_k/\F_{q^r})=\sum_{i | r} iB_i(F_k/\F_q),$$
then
$$\liminf_{k\rightarrow\infty} \sum_{i\mid r}\frac{iB_i(F_k/\F_{q})}{g_k}=
\liminf_{k\rightarrow\infty}\frac{B_1(G_k/\F_{q^r})}{g_k}=\mu_1({\mathcal G}/\F_{q^r}).$$
Let us consider the families
${\mathcal D}=( D_i )_{i\mid r}$ where $D_i=\{Q_{i,1}, \cdots, Q_{i,c_i}\}$ and where the $c_i$
elements $Q_{i,j}$ are formal symbols. Given a positive integer $n$, let us denote by $N_n({\mathcal D})$ 
the number of formal combinations 
$\sum_{i,j}\alpha_{i,j} Q_{i,j}$
such that all the $\alpha_{i,j}$ are positive and $\sum_i i\sum_j \alpha_{i,j}=n$.
For each $k$ we define  the two following families ${\mathcal D}_{k}$ and ${\mathcal D}'_{k}$:
\begin{itemize}
 \item for ${\mathcal D}_{k}$, $D_i$ is the set of places of degree $i$ of the function field
$F_k/\F_q$ and consequently $c_i(k)=B_i(F_k/\F_q)$.
\item for ${\mathcal D}'_{k}$, we choose $c'_1(k)=1$ and $c'_r(k)$ such that
$\sum_{i \mid r}ic_i(k) = rc'_r(k) +s_k+1$ with $0 \leq s_k \leq r-1$. For $1<i<r$
we choose $c_i(k)=0$. Then  the following holds:
$$c'_r(k)=\frac{1}{r}\left(\sum_{i \mid r}iB_i(F_k/\F_q)-s_k-1\right).$$ 
We note that 
$$\sum_{i|n}ic'_i(k)=1+rc'_r(k)\leq \sum_{i|r} ic_i,$$
namely the total number of points in $\cup D'_i$ is less than the number of points 
in $\cup D_i$ and their degrees are bigger. We conclude that 
$N_n({\mathcal D}_k) \geq N_n({\mathcal D'}_k)$.
\end{itemize}
We also have
$$\liminf_{k\rightarrow\infty}\frac{c'_r(k)}{g_k}=\frac{1}{r}\liminf_{k\rightarrow\infty}
\frac{\sum_{i|r} iB_i(F_k/\F_q)-s_k-1}{rg_k}=\frac{\mu_1({\mathcal G}/\F_{q^r})}{r}.$$
Using the Drinfeld Vladut bound on the constant field extension $(G_k/F_{q^r})_k$ of
$(F_k/F_q)_k$ we get
$$\limsup_{k\rightarrow\infty}\frac{c'_r(k)}{g_k}=\frac{1}{r}\limsup_{k\rightarrow\infty}
\frac{\sum_{i|r} iB_i(F_k/\F_q)-s_k-1}{rg_k}=$$
$$\frac{1}{r}\limsup_{k\rightarrow\infty}
\frac{B_1(G_k/\F_{q^r})-s_k-1}{rg_k}\leq
\frac{1}{r}\left (q^{\frac{r}{2}}-1\right).$$
Then we can apply the previous study with $\mu=\frac{\mu_1({\mathcal G}/\F_{q^r})}{r}$
to find an asymptotic lower bound on $N_n({\mathcal D'})_k$
and consequently an asymptotic lower bound on $N_n({\mathcal D})_k$ and on $h(F_k/\F_q)$.
\end{proof}

Note that Theorems \ref{mainun} and \ref{maindeux} apply under weaker assumptions 
than those of Theorem \ref{theoh2} due to Tsfasman \cite{tsfa} (cf also \cite{tsvlno} for a survey) since we do not need to know 
if the considered family of curves is asymptotically exact over $\F_q$ which is a very strongh property 
(not even to know if the limit $\lim_{k\rightarrow\infty} \frac{B_r(F_k/\F_q)}{g_k}$ exists).
 Moreover, if we except the case of recursive towers, the property to be asymptotically exact seems  
to be enough difficult to be proved 
 (cf. Remark 5.2 in \cite{tsvl} and \cite{baro4}). However, in the case where the family of 
 algebraic functions fields ${\mathcal F}/\F_q=(F_k/\F_q)_k$ satisfies:
 $$\liminf_{k\rightarrow\infty} \frac{B_r(F_k/\F_q)}{g_k}= \mu_r({\mathcal F}/\F_q)=\frac{1}{r}(q^{\frac{r}{2}}-1),$$
 which only is possible if $q^r$ is a perfect square,  it is known that the family then is asymptotically exact over $\F_q$. 
 More precisely, the family reaches then the Generalized Drinfeld-Vladut bound by attaining the 
 Drinfeld-Vladut of order $r$ \cite{baro4}.
 This allows us us to easily compare the bound obtained in Theorem  \ref{mainun} 
 with the limit reached by the class number deduced from 
 Theorem \ref{theoh2}. A direct estimation of this limit shows that it corresponds to the estimation given 
 by Theorem \ref{mainun}. Morever, in the case $r=1$, we also obtain the same estimation 
 as the estimation obtained in Theorem \ref{theoh1} which is a particular case not requiring the assumption 
 of an  asymptotically exact sequence. Note also that even if 
 $\lim_{k\rightarrow\infty} \sum_{i\mid r}\frac{iB_i(F_k/\F_{q})}{g_k}= \mu_1({\mathcal G}/\F_{q^r})= 
q^{\frac{r}{2}}-1$, which also is 
 only possible if $q^r$ is a perfect square, Theorem \ref{maindeux} is interesting since 
this condition 
 not implying the existence of limits $\lim_{k\rightarrow\infty} \frac{B_i(F_k/\F_{q})}{g_k}$, 
the assumptions of Theorem  \ref{theoh2}
 possibly are not satisfied.

\section{Examples} \label{descenttowergast}\label{sec:examples}
In this section, we study the class number of certain known towers ${\mathcal F}/\F_q=(F_k/F_q)$ 
of algebraic functions fields defined over a finite field $\F_q$.
Let $r\geq 1$ be an integer. As previously, we consider the limit $\mu_r({\mathcal F}/\F_q)$ for the tower ${\mathcal F}/\F_q$, defined as follows: 
$$\mu_r({\mathcal F}/\F_q)=\liminf_{k\rightarrow\infty} \frac{B_r(k)}{g_k}.$$ 

\subsection{Sequences ${\mathcal F}/\F_q$ with $\mu_r({\mathcal F}/\F_q)=\frac{1}{r}(q^{\frac{r}{2}}-1)$} 

In this section, for any prime power $q$ and for any interger $r\geq 1$ such that $q^r$ is a perfect square, 
we exhibit some examples of sequences of algebraic function fields defined over $\F_{q}$ 
reaching the generalized Drinfeld-Vladut bound of order $r$ \cite{baro4}.  

Moreover, we know accurate lower bounds on the number of places of concerned degree  $r$. 
Then, by using the results of 
the previous sections, we can give lower bounds on the class number for each step of these towers and 
also a more accurate lower bound on the asymptotical class number of these towers.

We consider the Garcia-Stichtenoth's tower ${T}_{0}/{\mathbb F}_{q^r}$ over ${\mathbb F}_{q^r}$ 
constructed in \cite{gast}. Recall that this tower
is defined recursively in the following way.
We set $F_1={\mathbb F}_{q^r}(x_1)$ the rational function field over  ${\mathbb F}_{q^r}$,
and for $i \geq 1$ we define 
$$F_{i+1}=F_i(z_{i+1}),$$
where $z_{i+1}$ satisfies the equation
$$z_{i+1}^{q^{\frac{r}{2}}}+z_{i+1}=x_i^{q^{\frac{r}{2}}+1},$$
with
$$x_i=\frac{z_i}{x_{i-1}} \hbox{ for } i\geq 2.$$

Let us denote by $g_k$ the genus of $F_k$ in $T_0/{\mathbb F}_{q^r}$. 
Let $T_1/{\mathbb F}_{q^{\frac{r}{2}}}=\left(G_i/{\mathbb F}_{q^{\frac{r}{2}}}\right)$ be the descent of the tower ${T}_{0}/{\mathbb F}_{q^r}$ 
on the finite field ${\mathbb F}_{q^{\frac{r}{2}}}$  and let $T_2/{\mathbb F}_{q}=\left(H_i/{\mathbb F}_{q}\right)$ be the descent of the tower ${T}_{0}/{\mathbb F}_{q^r}$ 
on the finite field ${\mathbb F}_{q}$, namely, for any integer $i$, $$F_i=G_i\otimes_{{\mathbb F}_{q^{\frac{r}{2}}}}{\mathbb F}_{q^{r}} \hbox{ and } F_i=H_i\otimes_{{\mathbb F}_{q}}{\mathbb F}_{q^{r}}.$$

Let us prove a proposition establishing 
that the tower $T_2/{\mathbb F}_q$ reaches the Generalized Drinfeld-Vladut Bound of order $r$.

\begin{proposition}\label{subfield}
Let $q$ be prime power and $r\geq 1$ an integer such that $q^r$ is a perfect square. 
 The tower $T_2/{\mathbb F}_q$ is such that 
 
 $$\lim_{k\rightarrow +\infty}\frac{B_r(H_k)}{g_k}=\frac{1}{r}(q^{\frac{r}{2}}-1).$$
 
\end{proposition}

\begin{proof}
First, note that as the algebraic function field $F_k/\F_{q^r}$ is a constant field extension of $G_k/{\mathbb F}_{q^{\frac{r}{2}}}$, above any place of degree one in $G_k/{\mathbb F}_{q^{\frac{r}{2}}}$
there exists a unique place of degree one in $F_k/\F_{q^r}$. Consequently, let us use
the classification given in \cite[p. 221]{gast} of the places
of degree one of $F_k/\F_{q^r}$. Let us remark that the number of places of degree one
which are not of type (A), is less or equal to $2q^r$ (see
\cite[Remark 3.4]{gast}).
Moreover, the genus $g_k$ of the algebraic function fields $G_k/{\mathbb F}_{q^{\frac{r}{2}}}$ and $F_k/\F_{q^r}$ is such that $g_k\geq q^k$
by \cite[Theorem 2.10]{gast}, then we can focus our study on places of type (A).
The places of type (A) are built recursively in the following
way (cf. \cite[p. 220 and Proposition 1.1 (iv)]{gast}).
Let $\alpha \in \F_{q^r}\setminus \{0\}$ and $P_\alpha$ be 
the place of $F_1/\F_{q^r}$ which is the zero of $x_1-\alpha$. For any $\alpha\in \F_{q^r}\setminus \{0\}$ the polynomial equation
$z_2^{q^{\frac{r}{2}}}+z_2=\alpha^{q^{\frac{r}{2}}+1}$ has $q^{\frac{r}{2}}$ distinct roots $u_1,\cdots u_{q^{\frac{r}{2}}}$ in $\F_{q^r}$,
and for each $u_i$ there is a unique place $P_{(\alpha,i)}$ of $F_2/\F_{q^r}$
above $P_\alpha$ and this place $P_{(\alpha,i)}$ is a zero of $z_2-u_i$.
We iterate now the process starting from the places $P_{(\alpha,i)}$
to obtain successively the places of type (A) of $F_3/\F_{q^r}, \cdots, F_k/\F_{q^r}, \cdots$;
then, each place $P$ of type (A) of $F_k/\F_{q^r}$ is a zero of $z_k-u$
where $u$ is itself a zero of $u^{q^{\frac{r}{2}}}+u = \gamma$ for some 
$\gamma \neq 0$ in $\F_{q^r}$. Let us denote by $P_u$ this place.
Now, we want to count the number of places $P'_u$ of degree one in $G_k/\F_{q^{\frac{r}{2}}}$, that is to say the only places 
which admit a unique place of degree one $P_u$ in $F_k/\F_{q^r}$ lying over $P'_u$.

First, note that it is possible only if  $u$ is a solution in $\F_{q^{\frac{r}{2}}}$ of the equation $u^{q^{\frac{r}{2}}}+u=\gamma$ where $\gamma$ is in $\F_{q^{\frac{r}{2}}}\setminus \{0\}$.
Indeed, if $u$ is not in $\F_{q^{\frac{r}{2}}}$, there exists an automorphism $\sigma$ in the Galois group $Gal(F_k/G_k)$ 
of the degree two Galois extension  $F_k/\F_{q^r}$ of $G_k/\F_{q^{\frac{r}{2}}}$ such that $\sigma(P_u)\neq P_u$.
Hence, the unique place of $G_k/\F_{q^{\frac{r}{2}}}$ lying under $P_u$ is a place of degree $2$. 


But
$u^{q^{\frac{r}{2}}}+u=\gamma$ has one solution in $\F_{q^{\frac{r}{2}}}$
if $p \neq 2$ and no solution in $\F_{q^{\frac{r}{2}}}$
if $p = 2$. Hence the number of places of degree one of $G_k/\F_{q^{\frac{r}{2}}}$
which are lying under a place of type (A) of $F_k/\F_{q^r}$ is equal to zero if $p=2$ and equal to ${q^{\frac{r}{2}}}-1$ if $p\neq2$.
We conclude that
$$\lim_{k\rightarrow +\infty} \frac{B_1(G_k/\F_{q^{\frac{r}{2}}})}{g_k}=0. $$
Let us remark that in any case, the number of places of degree one 
of $G_k/\F_{q^{\frac{r}{2}}}$ is less or equal to $2q^r$. 
Moreover, as $G_k/\F_{q^{\frac{r}{2}}}$ is the constant field extension of $H_k/\F_{q}$ of degree $\frac{r}{2}$, any place of  
$H_k/\F_{q}$  of degree $i<r$ dividing $r$ is totally decomposed in $G_k/\F_{q^{\frac{r}{2}}}$ and so is 
a place of $G_k/\F_{q^{\frac{r}{2}}}$ lying under only some places of degree one (at most $\frac{r}{2}$) belonging to $G_k/\F_{q^{\frac{r}{2}}}$.
Hence, we have for any integer $i$ dividing $r$ such that $i<r$, $$\lim_{k\rightarrow +\infty} \frac{B_i(H_k/\F_{q})}{g_k}=0. $$
Then, as the Garcia-Stichtenoth tower $T_0/\F_{q^r}$ attains the Drinfeld-Vladut bound \cite{gast}, we deduce the result 
by the relation $\sum_{i=1,i\mid r}^{r}iB_i(H_k/\F_{q})=B_r(F_k/\F_{q^r})$.

\end{proof}

\begin{theoreme}
Let $q$ be prime power and $r\geq 1$ an integer such that $q^r$ is a perfect square.
Let us consider the algebraic function fields  $F_{k}/{\mathbb F}_{q^r}$ and $H_{k}/{\mathbb F}_{q}$ 
of genus $g_k$ 
constituting respectively
the towers $T_0/\F_{q^r}$ and $T_2/{\mathbb F}_{q}$.
%
%
%
There are two positive real numbers $c$ and $c'$ such that the following holds:
for any $\epsilon >0$
there exists an integer $k_0$ such that
for any integer $k\geq k_0$, we have 
$$h(F_k/{\mathbb F}_{q^r})\geq c\left(\frac{q^r}{q^r-1}\right)^{B_1(F_k/{\mathbb F}_{q^r})}q^{rg_k},$$
where  $B_1(F_k/{\mathbb F}_{q^r})\geq (q^{r}-1).q^{\frac{r}{2}(k-1)}+2q^{\frac{r}{2}}$ and
$$h(H_k/{\mathbb F}_{q})\geq c'\left(\frac{q^r}{q^r-1}\right)^{B_r(H_k/{\mathbb F}_{q})}q^{g_k},$$ 
where $B_r(H_k/{\mathbb F}_{q})\geq \frac{g_k}{r}\left(q^{\frac{r}{2}}-1\right)(1-\epsilon)$ 
and 
$$g_k=\left\{
\begin{array}{ll}
(q^{r})^n+(q^{r})^{n-1}-(q^{r})^{\frac{n+1}{2}}-2(q^{r})^{\frac{n-1}{2}}+1 & \hbox{ if } n \hbox{ is odd },\\
(q^{r})^n+(q^{r})^{n-1}-\frac{1}{2}(q^{r})^{\frac{n}{2}+1}-\frac{3}{2}(q^{r})^{\frac{n}{2}}
-(q^r)^{\frac{n}{2}-1}+1 & \hbox{ if } n \hbox{ is even }.
\end{array} \right .
$$

\end{theoreme}

\begin{proof}
By a property of the Garcia-Stichtenoth tower \cite{gast} and by Proposition \ref{subfield}, we have
$$\beta_1(T_0/{\mathbb F}_{q^r})= 
\lim_{g_{k}\rightarrow +\infty} \frac{B_1(F_{k}/{\mathbb F}_{q^r})}{g_{k}} =q^{\frac{r}{2}}-1$$ 
and 
$$\beta_r(T_2/{\mathbb F}_q)= 
\lim_{g_{k}\rightarrow +\infty} \frac{B_r(H_{k}/{\mathbb F}_q)}{g_{k}} =
\frac{1}{r}(q^{\frac{r}{2}}-1).$$ 
Moreover, we have $B_1(F_k/{\mathbb F}_{q^r})\geq (q^{r}-1).q^{\frac{r}{2}(k-1)}+2q^{\frac{r}{2}}$ by 
\cite{gast}  and we have $B_1(H_k/{\mathbb F}_{q})>0$ by the proof of Theorem \ref{subfield}. 
Hence, we have also
$$\sum_{i\mid r}iB_i(H_k/\F_{q})\geq (q^{r}-1).q^{\frac{r}{2}(k-1)}+2q^{\frac{r}{2}}.$$ 
Then, we can use Theorem \ref{mainun} for the tower $T_0/\F_{q^r}$ 
and use Theorem \ref{maindeux}  for the tower $T_2/\F_{q}$ for assertion (2). 
\end{proof}

\subsection{Sequences ${\mathcal F}/\F_q$ with $\mu_r({\mathcal F}/\F_q)>0$ for certain integers $r\geq 1$} 

In this section, we study the class number of few towers defined over different finite fields $\F_q$, 
whose we only know that for a certain integer $r$, we have 
$\mu_r({\mathcal F}/\F_q)>0$ or $\mu_1({\mathcal F}/\F_{q^r})$. 
Hence, we only obtain asymptotical lower bounds on the class number of these towers. 
The studied towers are all tame towers exhibited by Garcia and Stichtenoth in \cite{gast2}.
Let us recall the recursive definition of an arbitrary  tower defined over $\F_q$.

\begin{definition}
A tower ${\mathcal T}$ is defined by the equation $\psi(y)=\phi(x)$  if $\psi(y)$ and $\phi(x)$ 
are two rational functions over $\F_q$ such that 
$${\mathcal T}=\F_q(x_0,x_1,x_2,...) \hbox{ with } \psi(x_{i+1})=\phi(x_i) \hbox{ for all } i\geq 0.$$
\end{definition}

\subsubsection{Examples 1: some tame towers of Fermat type}   

Let us recall some generalities about the towers of Fermat Type \cite{gast2}.

\begin{definition}
A tower ${\mathcal T}$ over $\F_q$ defined by the equation $$y^m=a(x+b)^m+c, \hbox{ with } (m,q)=1$$ 
is said to be a Fermat tower if for each $i\geq 0$, the field $\F_q$ is algebraically closed in $F_i$ and $[F_{i+1}:F_i]=m$.
\end{definition}

Now, we can give the following result:







\begin{proposition}
Let $l$ be a power of the characteristic of $\F_q$ and let $q=l^r$ with $r\geq 1$. 
Assume that $$r\equiv 0 \mod 2 \hbox{ or } l\equiv 0 \mod 2.$$
Then the equation $$y^{l-1}=-(x+b)^{l-1}+1, \hbox{ with } b\in \F_l^*$$ 
define a tower ${\mathcal F}/\F_l=(F_0,F_1,...F_n,...)$ over $\F_l$ such that 
for any $\alpha$ satisfying $0<\alpha<\frac{2}{r(l-2)}$, we have:
$$h(F_k/\F_l) \in\Omega\left(\left({\left(\frac{q}{q-1}\right)^{\alpha}l}\right)^{g_k}\right).$$
\end{proposition}

\begin{proof}
By Theorem 3.10 in \cite{gast2},  the equation $$y^{l-1}=-(x+b)^{l-1}+1, \hbox{ with } b\in \F_l^*$$ 
defines a Fermat tower ${\mathcal T}/\F_q=(T_0,T_1,...T_n,...)$ over $\F_q$ such that $\mu_1({\mathcal T}/\F_q)=\frac{2}{l-2}$.
Then the tower ${\mathcal F}/\F_l$ is the tower such that ${\mathcal T}=\F_q\otimes_{\F_l}{\mathcal F} $ is the constant field extension tower of 
${\mathcal F}$ of degree $r$. Hence, the tower 
${\mathcal F}/\F_l$ satisfies $\liminf_{k\rightarrow\infty} 
\sum_{i\mid r}\frac{iB_i(F_k/\F_{l})}{g_k}= \mu_1({\mathcal T}/\F_{q}) =\frac{2}{l-2}$ 
and we conclude by using Theorem \ref{maindeux}.
\end{proof}

We can obtain a similar result with other examples of Fermat towers exhibited in Section 3 of \cite{gast2}, in particular 
Example 3.12.

\subsubsection{Examples 2: some tame quadratic towers}   

Let us give few other examples of towers for which we are be able to give  very good asymptotics for the class number.

\begin{proposition}
Let us consider the tower  ${\mathcal F}/\F_3$ defined over $\F_3$ recursiveley by the equation $$y^2=\frac{x(x-1)}{x+1}.$$

Then the  tower ${\mathcal F}/\F_3=(F_0,F_1,...F_n,...)$ defined over $\F_3$  satisfies
$$h(F_k/\F_3) \in\Omega\left(\left(3{\left(\frac{9}{8}\right)^{\frac{1}{3}}}\right)^{g_k}\right).$$
\end{proposition}

\begin{proof}
By Example 4.3 in \cite{gast2},  the equation $y^2=\frac{x(x-1)}{x+1}$ 
defines a tower ${\mathcal T}/\F_9=(T_0,T_1,...T_n,...)$ over $\F_9$ such that $\mu_1({\mathcal T}/\F_9)\geq 2/3$.
Then the tower ${\mathcal F}/\F_3$ is the tower such that ${\mathcal T}=\F_9\otimes_{\F_3}{\mathcal F} $ is the constant field extension tower of 
${\mathcal F}$ of degree $r=2$. Hence, the tower ${\mathcal F}/\F_3$ satisfies $\liminf_{k\rightarrow\infty} \sum_{i\mid r}\frac{iB_i(F_k/\F_{l})}{g_k}= 
\mu_1({\mathcal T}/\F_{9}) \geq \frac{2}{3}$ and we conclude by using Theorem \ref{maindeux}.
\end{proof}

Now, we also have the following result:

\begin{proposition}
Let us consider the tower  ${\mathcal F}/\F_3$ defined over $\F_3$ recursively by the equation $$y^2=\frac{x(1-x)}{x+1}.$$

Then the  tower ${\mathcal F}/\F_3=(F_0,F_1,...F_n,...)$ defined over $\F_3$  satisfies
$$h(F_k/\F_3) \in\Omega\left(\left(3{\left(\frac{81}{80}\right)^{\frac{1}{2}}}\right)^{g_k}\right).$$
\end{proposition}

\begin{proof}
By Example 4.5 and Remark 4.6 in \cite{gast2},  the equation $y^2=\frac{x(1-x)}{x+1}$ 
defines a tower ${\mathcal T}/\F_{81}=(T_0,T_1,...T_n,...)$ over $\F_{81}$ such that $\mu_1({\mathcal T}/\F_{81})\geq 2$.
Then the tower ${\mathcal F}/\F_3$ is the tower such that ${\mathcal T}=\F_{81}\otimes_{\F_3}{\mathcal F} $ is the constant field extension tower of 
${\mathcal F}$ of degree $r=4$. Hence, the tower 
${\mathcal F}/\F_3$ satisfies $\liminf_{k\rightarrow\infty} 
\sum_{i\mid r}\frac{iB_i(F_k/\F_{l})}{g_k}= \mu_1({\mathcal T}/\F_{81}) \geq \frac{1}{2}$ 
and we conclude by using Theorem \ref{maindeux}.
\end{proof}

Then, we also have:

\begin{proposition}
Let us consider the tower  ${\mathcal F}/\F_5$ defined over $\F_5$ recursively by the equation $$y^2=\frac{x(x+2)}{x+1}.$$

Then the  tower ${\mathcal F}/\F_5=(F_0,F_1,...F_n,...)$ defined over $\F_3$  satisfies
$$h(F_k/\F_5) \in\Omega\left(\left(5{\left(\frac{25}{24}\right)^{\frac{1}{2}}}\right)^{g_k}\right).$$
\end{proposition}

\begin{proof}
By Example 4.8 and Remark 4.6 in \cite{gast2},  the equation $y^2=\frac{x(x+2)}{x+1}$ 
defines a tower ${\mathcal T}/\F_{25}=(T_0,T_1,...T_n,...)$ over $\F_{81}$ such that $\mu_1({\mathcal T}/\F_{25})\geq 1$.
Then the tower ${\mathcal F}/\F_5$ is the tower such that ${\mathcal T}=\F_{25}\otimes_{\F_5}{\mathcal F} $ is the constant field extension tower of 
${\mathcal F}$ of degree $r=2$. Hence, the tower 
${\mathcal F}/\F_5$ satisfies $\liminf_{k\rightarrow\infty} \sum_{i\mid r}\frac{iB_i(F_k/\F_{l})}{g_k}=
 \mu_1({\mathcal T}/\F_{25}) \geq \frac{1}{2}$ and we conclude by using Theorem \ref{maindeux}.
\end{proof}

\end{document}